\documentclass [smallcondensed]{svjour3}

\newtheorem{thm}{Theorem}[section]

\newtheorem{lem}[thm]{Lemma}

\newtheorem{defn}[thm]{Definition}

\usepackage{amsmath}
\usepackage{amssymb}

\usepackage{subfig}
\usepackage{siunitx}
\newcommand{\deltaLK}{\delta_{\ell,k}}

\ifx\pdfoutput\undefined
\usepackage{graphicx}
\else
\usepackage[pdftex]{graphicx}
\usepackage{epstopdf}
\fi

\begin{document}

\title{Wendland functions with increasing smoothness converge to a Gaussian \thanks{This work was supported by the Australian Research Council. The authors thank Holger Wendland, Robert Schaback and Simon Hubbert for helpful discussions. }}
\author{A. Chernih \and I. H. Sloan \and R. S. Womersley}

\institute{A. Chernih
\at School of Mathematics and Statistics, University of New South Wales, Sydney NSW 2052,
Australia\\Tel.: +61-410-697411, Fax: +612 93857123 \\\email{andrew@andrewch.com}
\and
I. H. Sloan
\at School of Mathematics and Statistics, University of New South Wales\\\email{i.sloan@unsw.edu.au}
\and
R. S. Womersley
\at School of Mathematics and Statistics, University of New South Wales\\\email{r.womersley@unsw.edu.au}
}

\maketitle
\begin{center}
\today
\end{center}

\begin{abstract}
The Wendland functions are a class of compactly supported radial basis functions with a user-specified smoothness parameter. We prove that with an appropriate rescaling of the variables, both the original and the ``missing" Wendland functions converge uniformly to a Gaussian as the smoothness parameter approaches infinity. We also explore the convergence numerically with Wendland functions of different smoothness.
\end{abstract}


%\begin{keywords}
%Radial basis functions, compact support, smoothness, Wendland functions, Gaussian.
%\end{keywords}

%\begin{AMS}
%33C90, 41A05, 41A15, 41A30, 41A63, 65D07, 65D10.
%\end{AMS}

\section{Introduction}

Radial basis functions (RBFs) are a popular tool for approximating scattered data and solving partial differential equations. Recent books covering practical and theoretical issues are Fasshauer \cite{Fas07} and Wendland \cite{Wen05}. A function $\Phi : \mathbb{R}^d \rightarrow \mathbb{R}$ is said to be \textit{radial} if there exists a function $\phi: [0,\infty) \rightarrow \mathbb{R}$ such that $\Phi(\mathbf{x}) = \phi(\|\mathbf{x}\|_2)$ for all $\mathbf{x} \in \mathbb{R}^d$. Then we can define an RBF for a given centre $\mathbf{x}_i \in \mathbb{R}^d$ as
\begin{equation*}
\Phi_i(\mathbf{x}) = \phi(\|\mathbf{x}-\mathbf{x}_i\|_2).
\end{equation*}

An interpolant $\mathcal{I}$ for the scattered data interpolation problem, where we are given data $(\mathbf{x}_j,y_j), j=1,\ldots,n$, with $\mathbf{x}_j \in \mathbb{R}^d$ and $y_j \in \mathbb{R}$, can be constructed as a linear combination
\begin{equation}
\mathcal{I}(\mathbf{x}) = \sum_{i=1}^n c_i \Phi_i(\mathbf{x}), \quad \mathbf{x} \in \mathbb{R}^d, \label{eqnInterpolation1}
\end{equation}
where the coefficients $c_1,\ldots,c_n$ are chosen so that
\begin{equation}
\mathcal{I}(\mathbf{x}_j) = y_j, \quad j=1,\ldots,n. \label{eqnInterpolation2}
\end{equation}
If $\Phi$ is positive definite, then \eqref{eqnInterpolation1} and \eqref{eqnInterpolation2} have a unique solution. We recall that a continuous function $f:\mathbb{R}^d \rightarrow \mathbb{R}$ is \textit{positive definite} (some would say \textit{strictly positive definite}) if for any $n$ distinct points $\mathbf{x}_1,\ldots,\mathbf{x}_n \in \mathbb{R}^d$, the quadratic form
\begin{equation*}
\sum_{i=1}^n \sum_{j=1}^n \alpha_i \alpha_j \, f(\mathbf{x}_i - \mathbf{x}_j)
\end{equation*}
is positive for all $\boldsymbol\alpha = [\alpha_1,\ldots,\alpha_n]^T \in \mathbb{R}^n \backslash \{\mathbf{0}\}$.

RBFs can be categorised as either \textit{globally supported} or \textit{compactly supported}. The first category includes Gaussians and multiquadrics \cite{Fas07}. Both have a scale parameter (also known as a shape or tension parameter), the selection of which is still a major ongoing research topic \cite{Rip99,FasZ07}.
 \newline

The second category includes Wendland functions \cite{Wen95}, Buhmann RBFs \cite{Buh03}, Wu's RBFs \cite{Wu95} and ``Euclid's hat" \cite{Gne99,Sch95}. In this paper we consider the Wendland functions, which are piecewise polynomial compactly supported functions. They have the minimum polynomial degree for any level of smoothness, and are positive definite since they have a strictly positive Fourier transform. The Wendland functions were originally derived for integer-order Sobolev spaces in odd dimensions in Wendland \cite{Wen95} and were then extended to even dimensions in Schaback \cite{Sch09} (Schaback called these the ``missing" Wendland functions). They are uniquely defined for a given spatial dimension $d$ and a smoothness parameter $k$ (up to a constant multiplier). All the Wendland functions are equal to zero outside [0,1].
\newline

Our purpose in this paper is to consider the limit of the original and the missing Wendland functions as the smoothness parameter $k$ goes to infinity. In Figure \ref{FigOrigWFsd3k1to5}, we can see the original Wendland functions in $\mathbb{R}^3$ for $k=1,\ldots,4$, where we have normalised the functions to have value 1 at the origin. The peak narrows as $k$ increases, demonstrating the need for a change of variable when considering the limit.

In Section \ref{SectionBackground} we define the original and misising Wendland functions, and give some needed background on Fourier transforms. Then in Section \ref{sectionNormEqualAreaWFs} we define the \textit{normalised equal area Wendland functions} $\psi_{\ell,k}$, obtained by a linear change of scale in the argument of the normalised functions that ensures equal areas under their graphs for all values of $k$. In Section \ref{SectionLimits} we consider the limit as the smoothness parameter goes to infinity. Our main theorem is Theorem \ref{thmMain}, which states that the normalised equal area Wendland functions converge uniformly to a Gaussian on the real half-line. Section \ref{SectionQualityApprox} gives numerical illustrations of the results.  A reasonable  conclusion from the experiments might be that there is little incentive to use high values for the smoothness parameter.

The similarity of appropriately scaled Wendland functions to a Gaussian has been mentioned in \cite{Mor08} and \cite{ForRS00}, in both cases for $\mathbb{R}^3$ with $k=1$.  No theoretical explanations were given in those papers.
\begin{figure}[!ht]
\centerline{
    \includegraphics[width=0.9\linewidth]{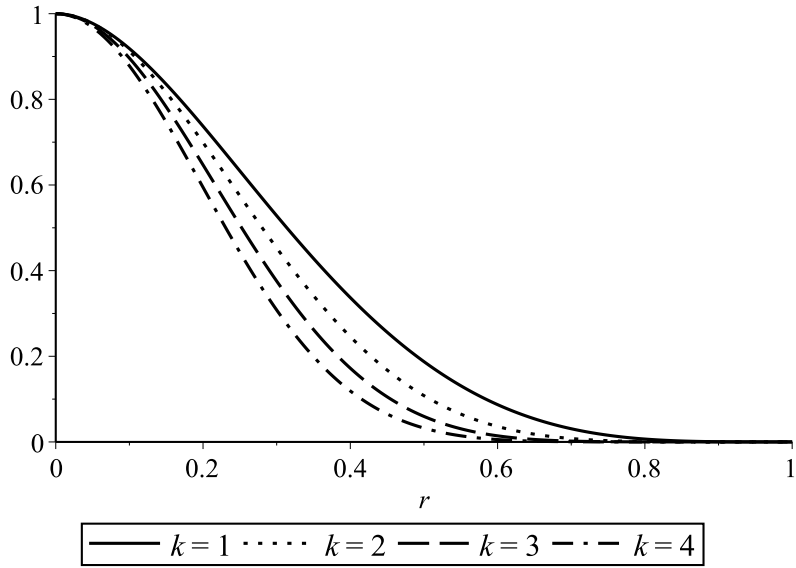}
        }
    \caption{The (original) Wendland functions $\phi_{\ell,k}(r)$ for $\ell = k + 2$ and $k = 1,\ldots,4$, normalised to have value 1 at $r=0$.}
    \label{FigOrigWFsd3k1to5}
\end{figure}

\section{Background} \label{SectionBackground} \hfill
\newline
In this section, we provide background material on the Wendland radial basis functions and Fourier transforms.
\subsection{Wendland functions} \hfill
\newline
Wendland functions were originally introduced in \cite{Wen95} and then more cases were added by Schaback in \cite{Sch09}. We shall refer to the functions from \cite{Wen95} as the \textit{original Wendland functions} and those  from \cite{Sch09} as the \textit{missing Wendland functions}. %A thorough investigation of both types of Wendland functions in terms of hypergeometric functions and other special functions is the focus of Hubbert in \cite{Hub10}.
We firstly define the Wendland functions and then discuss their properties.

Throughout the paper $k>0$ denotes the smoothness parameter of the Wendland functions, with $k$ an integer for the original Wendland functions and $k$ a half integer for the missing Wendland functions.  In limits $\lim_{k\to\infty}$ it is to be understood that $k$ goes to infinity through both integer and half integer values.  Also,  throughout the symbol $\ell$ will stand for
\begin{equation}\label{ldef}
\ell=\ell(k):=\left\lfloor \frac{d}{2}+k \right\rfloor+1
\end{equation}
where $d$ is the spatial dimension, and the floor function $\lfloor x \rfloor$ gives the largest integer less than or equal to $x$.

\begin{defn} With $d$ the spatial dimension, $2k$ a positive integer and $\ell$  given by \eqref{ldef}, let
\begin{equation}
\phi_{\ell,k}(r) := \left\{
 \begin{array}{ll}
 \displaystyle{\frac{1}{\Gamma(k) \, 2^{k-1}} \int_r^1 s \,(1-s)^{\ell} (s^2-r^2)^{k-1} \mathrm{d}s} & \displaystyle{\mathrm{for }\hspace{0.1in} 0 \leq r \leq 1}, \\ \\[-2ex]
0 & \displaystyle{\mathrm{for }\hspace{0.1in} r > 1.}
\end{array} \right.
 \label{formulaWFsIntDefn}
\end{equation}
This defines the original Wendland functions when $k$ is a positive integer and the missing Wendland functions when $k$ is a positive half-integer.
\end{defn}
Both the original Wendland functions and the missing Wendland functions are continuous on $[0,\infty)$. Note that the choice of $\ell$ given by \eqref{ldef} is the smallest integer that ensures that the resulting functions are positive definite. For fixed $d$ we have
\begin{equation}
\ell \sim k \ \mathrm{as} \ k \rightarrow \infty,  \label{LsimK}
\end{equation}
where $x \sim y$ denotes asymptotic equality, that is $\frac{x}{y}\rightarrow 1$.

We give explicit formulae for the original Wendland functions for $d=2,3$ and $k=1,\ldots,4$ (so $\ell=k+2$) in Table \ref{Wfpolys} where $\stackrel{\cdot}{=}$ denotes equality up to a positive constant factor. The support of all the original Wendland functions is $[0,1]$.
\begin{table}[!htbp]
\setlength{\extrarowheight}{5pt}
\begin{centering}
\begin{tabular}{|c|l|}
\hline
$k$ & Original Wendland function \\
\hline
1 & $\phi_{3,1}(r) \stackrel{\cdot}{=} (1-r)^4_+ (4r+1)$ \\
2 & $\phi_{4,2}(r) \stackrel{\cdot}{=} (1-r)^6_+ (35r^2 + 18r+3)$ \\
3 & $\phi_{5,3}(r) \stackrel{\cdot}{=} (1-r)^8_+ (32r^3+25r^2+8r+1)$ \\
4 & $\phi_{6,4}(r) \stackrel{\cdot}{=} (1-r)^{10}_+ (429r^4 + 450r^3 + 210r^2 + 50r + 5)$ \\

\hline
\end{tabular} \caption{The original Wendland functions $\phi_{\ell,k}(r)$ for $d=2,3$ and $k=1,\ldots,4$. } \label{Wfpolys}
\end{centering}
\end{table}

Schaback \cite{Sch09} extended Wendland's original approach to introduce the missing Wendland functions.  An important diffference between the original Wendland functions and the missing Wendland functions is that the missing Wendland functions, whilst still being supported on $[0,1]$, now have logarithmic and square-root multipliers of polynomial components. We give explicit formulae for the missing Wendland functions for $d=2$ and $k=\frac{1}{2},\frac{3}{2}$ and $\frac{5}{2}$ in Table \ref{missWfpolys}, where \begin{equation*}
L(r) := \log \left(\frac{r}{1+\sqrt{1-r^2}} \right) \quad \mbox{and} \quad S(r) := \sqrt{1-r^2}, \quad r \in (0,1].
\end{equation*}
\begin{table}[!htbp]
\setlength{\extrarowheight}{5pt}
\begin{centering}
\begin{tabular}{|c|l|}
\hline
$k$ & Missing Wendland function \\
\hline
$\frac{1}{2}$ & $\phi_{2,\frac{1}{2}}(r) \stackrel{\cdot}{=} 3r^2L(r) + (2r^2+1)S(r)$ \\
$\frac{3}{2}$ & $\phi_{3,\frac{3}{2}}(r) \stackrel{\cdot}{=} -15r^4(6+r^2)L(r) - (81r^4+28r^2-4)S(r)$ \\
$\frac{5}{2}$ & $\phi_{4,\frac{5}{2}}(r) \stackrel{\cdot}{=} (945r^8+2520r^6)L(r)+(256r^8+2639r^6+690r^4-136r^2+16)S(r)$ \\
\hline
\end{tabular} \caption{The missing Wendland functions $\phi_{\ell,k}(r)$ for $d=2$ and $k=\frac{1}{2},\frac{3}{2},\frac{5}{2}$. } \label{missWfpolys}
\end{centering}
\end{table}

Hereafter by \textit{Wendland functions} we will mean both the original and missing Wendland functions. Both will be denoted by $\phi_{\ell,k}(r)$.

\subsection{Fourier transforms} \hfill
The subsequent proofs rely heavily on Fourier transforms. This section provides definitions and outlines key properties. For further information, see \cite{SteW71,Wen05}.

The Fourier transform of $f \in L_1(\mathbb{R}^d)$ is defined as
\begin{equation*}
\widehat{f}(\mathbf{\mathbf{z}}) := (2\pi)^{-\frac{d}{2}} \int_{\mathbb{R}^d} f(\mathbf{x}) \, e^{-i\mathbf{\mathbf{z}} \cdot \mathbf{x}} \, \mathrm{d} \mathbf{x}, \hspace{0.1in} \mathbf{z} \in \mathbb{R}^d.
\end{equation*}
For the case of a radial function $\Phi \in L_1(\mathbb{R}^d), \Phi(\mathbf{x})=\phi(\|\mathbf{x}\|)$, with $\|\cdot\|$ denoting the Euclidean norm in $\mathbb{R}^d$, the Fourier transform is also radial, and is given by $\widehat{\Phi}(\mathbf{z}) = \mathcal{F}_d \phi(\|\mathbf{z}\|)$ where
\begin{equation}
\mathcal{F}_d \phi(z) := z^{1-\frac{d}{2}} \int_0^{\infty} \phi(y) \, y^{\frac{d}{2}} \, J_{\frac{d}{2}-1}(z\,y) \, \mathrm{d} y, \label{eqnRadialFourierTrans}
\end{equation}
and $J_{\nu}(y)$ denotes the Bessel function of the first kind of order $\nu$. In particular, from the properties
\begin{equation}\label{Bessel}
|J_\nu(x)| \leq \frac{|x|^\nu}{2^\nu\Gamma(\nu+1)}, \, x \in \mathbb{R},\quad |x|^{-\nu} J_\nu(x)\to  \frac{1}{2^\nu\Gamma(\nu+1)}\mbox{ as }x\to 0,
\end{equation}
for $\nu \geq -\frac{1}{2}$, it follows easily that
\begin{equation}
\mathcal{F}_d \phi(0) := \frac{1}{2^{\frac{d}{2}-1}\Gamma\left(\frac{d}{2}\right)} \int_0^{\infty} \phi(y) \, y^{d-1} \, \mathrm{d} y. \label{eqnRadialFourierTransZero}
\end{equation}

From the Fourier inversion theorem applied to radial functions, we know that if $\Phi \in L_1(\mathbb{R}^d)$ with $\Phi(\mathbf{x}) = \phi(\|\mathbf{x}\|), \phi:[0,\infty) \rightarrow \mathbb{R}$, and if $\widehat{\Phi} \in L_1(\mathbb{R}^d)$, then
\begin{equation}
\phi(y) = y^{1-\frac{d}{2}} \int_0^{\infty} \mathcal{F}_d\phi(z) \, z^{\frac{d}{2}} \, J_{\frac{d}{2}-1}(yz) \, \mathrm{d} z. \label{eqnHankelInversion}
\end{equation}

We also recall that if $f \in L_1(\mathbb{R}^d)$ is continuous at zero and positive definite then its Fourier transform is in $L_1(\mathbb{R}^d)$ and is non-negative \cite{SteW71}.

The Gaussian radial basis function with scale parameter $\alpha>0$, which we denote by $G_{\alpha}(y)$, is
\begin{equation}\label{G}
G_{\alpha}(y) := e^{-\alpha \, y^2}, \quad y \in \mathbb{R},
\end{equation}
and its Fourier transform is given by
\begin{equation}\label{FTG}
\widehat{G}_{\alpha}(z) = \frac{1}{(2\alpha)^\frac{d}{2}} \, e^{-\frac{z^2}{4 \alpha}}, \quad z \in \mathbb{R}.
\end{equation}

We first define the generalised hypergeometric function and then state the Fourier transform of the Wendland functions. Further details on generalised hypergeometric functions can be found in \cite{Abr72} and \cite{And00}.
\newline
\begin{defn}
The generalised hypergeometric function ${}_pF_q(a_1,\ldots,a_p;b_1,\ldots,b_q;x)$ is
\begin{equation*}
{}_pF_q(a_1,\ldots,a_p;b_1,\ldots,b_q;x) := \sum_{n=0}^{\infty} \frac{(a_1)_n \cdots (a_p)_n}{(b_1)_n \cdots (b_q)_n }\frac{x^n}{n!},
\end{equation*}
where none of $b_1,\ldots,b_q$ is a negative integer or zero and where
\begin{equation}
(c)_n := c(c+1)\cdots(c+n-1) = \frac{\Gamma(c+n)}{\Gamma(c)}, \hspace{0.1in} n \geq 1
\end{equation}
denotes the Pochhammer symbol, with $(c)_0 = 1$.
\end{defn}

When $p \leq q$ the series converges for all finite $x$ and defines an entire function. When $p=q+1$ the series converges absolutely for $|x|<1$, and also at $x=1$ if
$$\sum_{i=1}^q b_i - \sum_{i=1}^p a_i > 0.$$

\begin{lem}
\label{thmFTwf}
The $d$-dimensional Fourier transform of
$\phi_{\ell,k}$ is given by
\begin{equation*}
\mathcal{F}_d \phi_{\ell,k}(z) =C_{d}^{\ell,k}
{}_1F_2\left(\frac{d+1}{2}+k; \frac{\ell+d+1}{2}+k,
\frac{\ell+d+2}{2}+k; -\frac{z^2}{4} \right), \quad z \geq 0,
\end{equation*}
where
\[
C_{d}^{\ell,k}:=\frac{2^{k+\frac{d}{2}}\Gamma(\ell+1)\Gamma\left(\frac{d+1}{2}+k\right)}{\sqrt{\pi}\Gamma(\ell+d+2k+1)}.
\]
\end{lem}
\begin{proof}
The proof can be found in \cite[Theorem 3]{Zas06}.
\end{proof}

\section{Normalised equal area Wendland functions} \label{sectionNormEqualAreaWFs} \hfill
\newline
Hubbert \cite{Hub10} expresses the Wendland functions in terms of Legendre functions. Equation (3.4) in \cite{Hub10} states that for $r \in (0,1]$
\begin{equation}
\phi_{\ell,k}(r) = \frac{ \Gamma(\ell+1)}{2^{\ell+k} \Gamma(\ell+k+1)}\,(1-r^2)^{\ell+k} r^{-\ell} {}_2F_1\left(\frac{\ell}{2},k + \frac{\ell+1}{2}; \ell+k+1;1-\frac{1}{r^2}\right). \label{OrigEqn}
\end{equation}
Now we apply the identity \cite[15.3.4]{Abr72}
\begin{equation}
{}_2F_1(a,b;c;z) = (1-z)^{-a} {}_2F_1\left(a,c-b;c;\frac{z}{z-1}\right) \label{F3}
\end{equation}
to (\ref{OrigEqn}), which gives us, for $r \in [0,1]$,
\begin{equation}
\phi_{\ell,k}(r) = \frac{\Gamma(\ell+1)}{2^{\ell+k} \Gamma(\ell+k+1)} (1-r^2)^{\ell+k} {}_2F_1 \left(\frac{\ell}{2},\frac{\ell+1}{2};\ell+k+1;1-r^2 \right), \label{EqnStep1}
\end{equation}
where we recover the case of $r=0$ by right continuity.
We will need to normalise the Wendland functions, and thus need the value of $\phi_{\ell,k}(0)$.

\begin{lem}\label{Phi_zero}
\begin{equation}
\phi_{\ell,k}(0) = \frac{\Gamma(\ell+1)\Gamma(2k)}{2^{k-1}\Gamma(k)\Gamma(\ell+2k+1)}. \label{phiZero}
\end{equation}
\end{lem}
\begin{proof}
To calculate $\phi_{\ell,k}(0)$ we need the value of the hypergeometric function in \eqref{EqnStep1} at the argument 1 (since in $\phi_{\ell,k}(r)$ the hypergeometric function has argument $1-r^2$). From \cite[15.1.20]{Abr72} we have the identity
\begin{equation}
{}_2F_1(a,b;c;1) = \frac{\Gamma(c)\Gamma(c-b-a)}{\Gamma(c-b)\Gamma(c-a)} , \hspace{0.1in} c-b-a > 0. \label{FvalOne}
\end{equation} Applying \eqref{FvalOne} to \eqref{EqnStep1} shows that
$$ \phi_{\ell,k}(0) = \frac{\Gamma(\ell+1)\,\Gamma(k+\frac{1}{2})}{2^{\ell+k}\,\Gamma(\frac{\ell}{2}+k+\frac{1}{2})\,\Gamma(\frac{\ell}{2}+k+1)}.$$
Using the duplication formula for the gamma function \cite[6.1.18]{Abr72}
\begin{equation}
\Gamma(z)\Gamma \left(z+\frac{1}{2}\right) = 2^{1-2z} \sqrt{\pi}\, \Gamma(2z), \label{DoubleGamma}
\end{equation} twice -- firstly for $\Gamma(\frac{\ell}{2}+k+\frac{1}{2})\Gamma(\frac{\ell}{2}+k+1)$ and then for $\Gamma(k)\Gamma\left(k+\frac{1}{2}\right)$ -- and with several terms cancelling out in the numerator and denominator, we get the desired result. \hspace*{\fill} \qed
\end{proof}

We will also need the following result for the area under the Wendland functions.
\begin{lem}\label{areaThm}
\begin{equation}
\int_0^\infty\phi_{\ell,k}(r)\mathrm{d} r = \frac{2^{k} \, \Gamma(\ell+1) \, \Gamma(k+1)}{\Gamma(\ell+2k+2)}. \label{area}
\end{equation}
\end{lem}
\begin{proof}
This follows from \eqref{eqnRadialFourierTransZero} and Lemma \ref{thmFTwf} on setting $d=1$ (noting that $\phi_{\ell,k}$ has no explicit $d$ dependence). \hspace*{\fill} \qed
\end{proof}

Now we can define the \textit{normalised equal area Wendland functions} $\psi_{\ell,k}$. These are Wendland functions normalised to have the value $1$ at $0$, and with a change of scale in the argument so that the normalised equal area Wendland functions have integrals over the real half-line equal to the integral of $\exp(-\alpha y^2)$, where $\alpha>0$ can be chosen for the convenience of the user.
\begin{thm}
\label{PsiTheorem}
With $\alpha$ an arbitrary positive real number,
the normalised equal area Wendland functions are given by
\begin{equation}
\psi_{\ell,k}(y) = \frac{2^{k-1}\Gamma(k)\Gamma(\ell+2k+1)}{\Gamma(\ell+1)\Gamma(2k)}\left\{
\begin{array}{ll}
\phi_{\ell,k}\left(\frac{y}{\deltaLK(\alpha)}\right) & \mathrm{ for }\hspace{0.1in} 0 \leq y \leq \deltaLK(\alpha), \\
0 & \mathrm{ for }\hspace{0.1in} y > \deltaLK(\alpha)
\label{psiFormula}
\end{array} \right.,
\end{equation}
where
\begin{equation}
\deltaLK(\alpha) := \frac{(\ell + 2k+1) \, \Gamma\left(k+\frac{1}{2}\right)}{2 \, \sqrt{\alpha} \, \Gamma(k+1)}. \label{DeltaA}
\end{equation}
\end{thm}
\begin{proof}
It follows from Lemmas \ref{areaThm} and \ref{Phi_zero} and the duplication formula \eqref{DoubleGamma} that the integral of $\psi_{\ell,k}$ equals $\sqrt{\pi}/(2\sqrt{\alpha})$, which is the integral of $\exp{(-\alpha y^2)}$ over the half line, and that $\psi_{\ell,k}(0)=1$. \hspace*{\fill}\qed
\end{proof}

In Figure \ref{FigPsid3k1to5} we plot the normalised equal area Wendland functions $\psi_{\ell,k}$ for $d=2,3$,$k=1,\ldots,4$ with $\alpha=1$.
\begin{figure}[!ht]
\centerline{
    \includegraphics[width=0.9\linewidth]{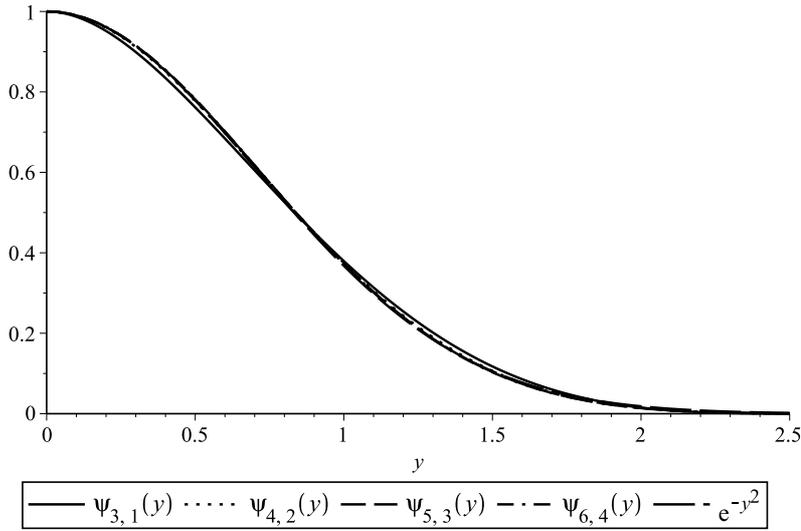}
    }
    \caption{The normalised equal area Wendland functions $\psi_{3,1}(y)$, $\psi_{4,2}(y)$, $\psi_{5,3}(y)$, $\psi_{6,4}(y)$ with $\alpha=1$ and $\exp(-y^2).$ }
    \label{FigPsid3k1to5}
\end{figure}
\newline

We also need the results in the next two lemmas.

\begin{lem}
\label{lemmaUpperBoundDelta}
For $k \geq \min \left( \frac{d}{2},1 \right)$,
\begin{equation}
\delta_{\ell,k}(\alpha) \leq \frac{3 \sqrt{k}}{\sqrt{\alpha}}. \label{eqnUppBoundDelta}
\end{equation}
\end{lem}
\begin{proof}
From \cite{Wen48} we have the double inequality
\begin{equation*}
\left(\frac{x}{x+s} \right)^{1-s} \leq \frac{\Gamma(x+s)}{x^s \, \Gamma(x)} \leq 1,
\end{equation*}
for $0 < s < 1$ and $x>0$. With $s=\frac{1}{2}$ and using $\Gamma(k+1) = k \Gamma(k)$, this gives
\begin{equation*}
\delta_{\ell,k}(\alpha) = \frac{(\ell+2k+1)}{2 \sqrt{\alpha}}\frac{\Gamma\left(k+\frac{1}{2}\right)}{\Gamma(k+1)}\\
\leq \frac{(\ell+2k+1)}{2\sqrt{k \, \alpha}} \leq \frac{3\sqrt{k}}{\sqrt{\alpha}}.\hspace*{\fill}
\end{equation*}
\hspace*{\fill} \qed
\end{proof}

\begin{lem}
\label{lemmaRatioGammas}
Let $\eta>0$. The function $f_{\eta}: (0,\infty) \rightarrow \mathbb{R}$ defined by
\begin{equation}
f_{\eta}(y) := \frac{\Gamma(y + \eta)}{\Gamma(y)}, \quad y > 0 \label{eqnUppBoundGammas},
\end{equation}
is increasing on $(0,\infty)$.
\end{lem}
\begin{proof}
Defining $F_{\eta}(y) := \log f_{\eta}(y)$, it is clear that $f_\eta$ is increasing on $(0,\infty)$ if and only if $F_\eta$ is increasing. But
\begin{equation*}
\frac{\mathrm{d}\,F_{\eta}(y)}{\mathrm{d}y}
= \frac{\mathrm{d}\,\log\,\Gamma(y+\eta)}{\mathrm{d}y}-\frac{\mathrm{d}\,\log\,\Gamma(y)}{\mathrm{d}y}
=\psi_0(y+\eta) - \psi_0(y),
\end{equation*}
where $\psi_0:=\mathrm{d}\,\log\,\Gamma(y)/\mathrm{d}y$ is the digamma function.  Since the digamma function is increasing on $(0,\infty)$, the result follows.    \hspace*{\fill} \qed
\end{proof}

\section{Limit of the Wendland functions as $k \rightarrow \infty$} \label{SectionLimits} \hfill
\newline
In this section we derive the limit of the normalised equal area Wendland functions as $k \rightarrow \infty$. We start with a convergence result for the Fourier transforms.

\begin{thm}
\label{thmConvFourier}
Let $\alpha$ be a positive real constant, and let $\psi_{\ell,k}$ be the normalised equal area Wendland functions defined by \eqref{psiFormula} and \eqref{DeltaA} with $\ell$ given by \eqref{ldef}. Then
\begin{equation}
\lim_{k \rightarrow \infty} \mathcal{F}_d\psi_{\ell,k}(z) = \widehat{G}_{\alpha}(z)  \label{eqnConvFourier}
\end{equation}
uniformly for $z$ in an arbitrary bounded subinterval of $\mathbb{R}^+$.
\end{thm}
\begin{proof}
Firstly, we express the Fourier transform of $\psi_{\ell,k}$ in terms of the Fourier transform of $\phi_{\ell,k}$. Writing $\delta_{\ell,k}$ for $\delta_{\ell,k}(\alpha)$ and using the transformation $y = r \, \delta_{\ell,k}$ together with \eqref{eqnRadialFourierTrans} and Theorem \ref{PsiTheorem}, gives
\begin{eqnarray}
\mathcal{F}_d \psi_{\ell,k}(z) &=&  \frac{\delta_{\ell,k}^d 2^{k-1}\Gamma(k)\Gamma(\ell+2k+1)}{\Gamma(\ell+1)\Gamma(2k)} \, \left( \mathcal{F}_d \phi_{\ell,k} \right) (\delta_{\ell,k} z) \nonumber\\
&=& \frac{\delta_{\ell,k}^d \, 2^{-\frac{d}{2}}\Gamma(\ell+2k+1)\Gamma(k)\Gamma(d+2k)}{\Gamma(2k)\Gamma\left(\frac{d}{2}+k\right)\Gamma(\ell+d+2k+1)} \times \nonumber\\&& {}_1F_2 \left(\frac{d+1}{2} + k ; \frac{\ell+d+2k+1}{2}, \frac{\ell+d+2k+2}{2}; -\frac{\delta_{\ell,k}^2 z^2}{4} \right) \nonumber\\
&=& 2^{-\frac{d}{2}} \sum_{n=0}^{\infty} \frac{\Gamma(d+2k+2n)\Gamma(\ell+2k+1)\Gamma(k)}{\Gamma(2k)\Gamma(\ell+2k+1+d+2n)\Gamma\left(k+\frac{d}{2}+n\right)} \, \delta_{\ell,k}^{d+2n} \frac{\left(-\frac{z^2}{4} \right)^n}{n!} \nonumber \\
&=& 2^{-\frac{d}{2}} \sum_{n=0}^{\infty} w_n(k) \left(-\frac{z^2}{4} \right)^n, \label{eqnFTpsi2}
\end{eqnarray}
where
\begin{equation}
w_n(k) := \frac{\Gamma(d+2k+2n)\Gamma(\ell+2k+1)\Gamma(k)}{\Gamma(2k)\Gamma(\ell+2k+1+d+2n)\Gamma\left(k+\frac{d}{2}+n\right)} \, \frac{\delta_{\ell,k}^{d+2n}}{n!}.
\end{equation}

Using \eqref{DoubleGamma} repeatedly, together with Lemma \ref{lemmaUpperBoundDelta} and the following inequalities (both a consequence of Lemma \ref{lemmaRatioGammas})
\begin{equation*}
\frac{\Gamma(\ell+2k+1)}{\Gamma(\ell+2k+1+d+2n)} \leq \frac{\Gamma(3k)}{\Gamma(3k+d+2n)},\quad
\frac{\Gamma \left( \frac{d}{2}+k+n+\frac{1}{2} \right)}{\Gamma \left( k + \frac{1}{2} \right)} \leq \frac{\Gamma \left( \frac{d}{2} + \frac{3k}{2} + n \right)}{\Gamma \left( \frac{3k}{2} \right)},
\end{equation*}
we have for $k \geq 1$
\begin{eqnarray*}
w_n(k) &=& \frac{ 2^{d+2n} \Gamma\left(\frac{d}{2}+k+n+\frac{1}{2} \right) \Gamma(\ell+2k+1)}{\Gamma\left(k + \frac{1}{2} \right) \Gamma(\ell+2k+1+d+2n)} \,  \frac{\delta_{\ell,k}^{d+2n}}{n!} \\
&\leq& \frac{ 2^{d+2n} \Gamma\left(\frac{d}{2}+\frac{3k}{2}+n \right) \Gamma(3k)}{\Gamma\left(\frac{3k}{2} \right) \Gamma(3k+d+2n)} \,  \frac{\delta_{\ell,k}^{d+2n}}{n!} \\
%&=& \frac{ 2^{d+2n} \Gamma\left(\frac{d}{2}+\frac{3k}{2}+n \right) 2^{3k-1} \Gamma\left(\frac{3k}{2} \right) \Gamma \left( \frac{3k+1}{2}\right)}{\Gamma\left(\frac{3k}{2} \right) 2^{3k+d+2n-1} \Gamma\left(\frac{3k+d}{2}+n \right) \Gamma \left( \frac{3k+d+1}{2}+n\right)} \,  \frac{\delta_{\ell,k}^{d+2n}}{n!} \\
&=& \frac{\Gamma\left(\frac{3k+1}{2} \right)}{\Gamma\left( \frac{3k+d+1}{2}+n\right)} \,  \frac{\delta_{\ell,k}^{d+2n}}{n!}
\leq \frac{1}{\left(\frac{3k}{2}\right)^{\frac{d}{2}+n}} \, \left(\frac{3\sqrt{k}}{\sqrt{\alpha}}\right)^{d+2n} \, \frac{1}{n!}\\
&=& \left(\frac{6}{\alpha}\right)^{\frac{d}{2}+n} \, \frac{1}{n!}
=: U_n,
\end{eqnarray*}
where in the penultimate step we used the bound \cite[5.6.8]{DLMF}
\begin{equation}
\frac{\Gamma(x+a)}{\Gamma(x+b)} \leq \frac{1}{x^{b-a}}, \hspace{0.1in} x > 0, b-a\geq 1, a \geq 0.
\end{equation}

The ratio test shows that $\sum_n U_n \,\left(-\frac{z^2}{4} \right)^n$ is absolutely convergent for all $z\in \mathbb{R}^+$. Therefore by the dominated convergence theorem we can take the limit as $k \rightarrow \infty$ inside the infinite sum in \eqref{eqnFTpsi2}, giving
\begin{equation*}
\lim_{k \rightarrow \infty} \mathcal{F}_d \psi_{\ell,k}(z) = 2^{-\frac{d}{2}} \sum_{n=0}^{\infty} \lim_{k \rightarrow \infty} w_n(k) \, \left(-\frac{z^2}{4} \right)^n  = \left(2 \alpha \right)^{-\frac{d}{2}} \sum_{n=0}^{\infty} \frac{\left(-\frac{z^2}{4\alpha} \right)^n }{n!}= \widehat{G}_{\alpha}(z),
\end{equation*}
where we used \eqref{FTG} and the asymptotic equality, see \cite{DLMF},
\begin{equation}
\frac{\Gamma(x+a)}{\Gamma(x+b)} \sim x^{a-b}. \label{limRatioGammas}
\end{equation}
This proves pointwise convergence of a sequence of continuous functions, which is necessarily uniform on a compact interval. \hspace*{\fill} \qed \newline
\end{proof}

We are now ready to state the main result. \newline
\begin{thm}
\label{thmMain}
Let $\alpha$ be a positive number and $\psi_{\ell,k}$ be the normalised equal area Wendland functions defined by \eqref{psiFormula}, \eqref{DeltaA} and \eqref{ldef}. Then
\begin{equation}
\lim_{k \rightarrow \infty} \psi_{\ell,k}(y) = G_{\alpha}(y),  \label{mainTheorem2}
\end{equation}
with uniform convergence for $y \in \mathbb{R}^+$.
\end{thm}
\begin{proof}
It follows from \eqref{eqnHankelInversion}, with the aid of \eqref{Bessel}, that for arbitrary $y,Z\in \mathbb{R}^+$
\begin{eqnarray}
\frac{\Gamma(\frac{d}{2})}{2^{1-\frac{d}{2}}}|\psi_{\ell,k}(y)-G_{\alpha}(y)| &=& \frac{\Gamma(\frac{d}{2})}{2^{1-\frac{d}{2}}}y^{1-\frac{d}{2}}\Bigg|\int_0^{\infty} \left(\mathcal{F}_d \psi_{\ell,k}(z) - \widehat{G}_{\alpha}(z)\right)z^{\frac{d}{2}} \, J_{\frac{d}{2}-1}(yz) \, \mathrm{d}z\Bigg| \nonumber \\
&\leq&  \int_0^{\infty} |\mathcal{F}_d \psi_{\ell,k}(z) - \widehat{G}_{\alpha}(z)|z^{d-1} \, \mathrm{d}z\nonumber\\
&\leq& \int_0^Z |\mathcal{F}_d \psi_{\ell,k}(z) - \widehat{G}_{\alpha}(z)|z^{d-1} \, \mathrm{d}z + \int_Z^{\infty} \mathcal{F}_d \psi_{\ell,k}(z) \, z^{d-1} \, \mathrm{d}z + \int_Z^{\infty} \widehat{G}_{\alpha}(z) \, z^{d-1} \, \mathrm{d}z \nonumber \\
&=& \int_0^Z |\mathcal{F}_d \psi_{\ell,k}(z) - \widehat{G}_{\alpha}(z)|z^{d-1} \, \mathrm{d}z + \int_0^Z \left( \widehat{G}_{\alpha}(z) - \mathcal{F}_d \psi_{\ell,k}(z) \right) z^{d-1} \, \mathrm{d}z \nonumber \\
 && \hspace{2.5in} + 2 \int_Z^{\infty} \widehat{G}_{\alpha}(z) \, z^{d-1} \, \mathrm{d}z \nonumber \\
&\leq& 2 \int_0^Z |\mathcal{F}_d \psi_{\ell,k}(z) - \widehat{G}_{\alpha}(z)|z^{d-1} \, \mathrm{d}z + 2\int_Z^{\infty} \widehat{G}_{\alpha}(z) \, z^{d-1} \, \mathrm{d}z, \label{eqnBoundDiffTail}
\end{eqnarray}
where we used the positivity of $\mathcal{F}_d \psi_{\ell,k}$ and $\widehat{G}_\alpha$, and
\begin{equation*}
\int_0^{\infty} \mathcal{F}_d \psi_{\ell,k}(z) \, z^{d-1} \, \mathrm{d} z = \int_0^{\infty} \widehat{G}_{\alpha}(z) \, z^{d-1} \, \mathrm{d}z = 2^{\frac{d}{2}-1} \Gamma\left(\frac{d}{2} \right),
\end{equation*}
which follow from \eqref{eqnRadialFourierTransZero} with $\mathcal{F}_d$ replaced by $\mathcal{F}_d^{-1}$.

Since the bound is independent of $y$, the result now follows from the integrability of $\widehat{G}_\alpha(z)z^{d-1}$ over $\mathbb{R}^+$, together with the uniform convergence property established in the preceding theorem.
\hspace*{\fill} \qed \newline
\end{proof}

\section{Numerical results} \label{SectionQualityApprox} \hfill
\newline
In this section we present numerical results regarding the differences between the appropriately scaled Wendland functions and the Gaussian limit established in Theorem \ref{thmMain}. We also consider an interpolation example using both the Wendland functions $\phi_{\ell,k}$ and the normalised equal area Wendland functions $\psi_{\ell,k}$.

\subsection{Difference with the limiting Gaussian} \hfill
\newline
Let the differences between the normalised equal area Wendland functions and the limiting Gaussian be
\begin{equation*}
E_{\ell,k}(y) := \psi_{\ell,k}(y)-\exp(-\alpha y^2)
\end{equation*}
and let
\begin{equation*}
\epsilon_{\ell,k} := \sup_{y \geq 0} \left|E_{\ell,k}(y)\right|.
\end{equation*}
Note that the change of variable used to define $\psi_{\ell,k}$ depends on the parameter $\alpha$.
Figure \ref{wfMinusGaussianGraphs} shows plots of $E_{\ell,k}(y)$ with $\alpha=1$. The upper plots are for $d=2$, and show $k=1.5, \ell=3$ and $k=5.5, \ell=7$ respectively.  The lower plots are for $d=3$, and show $k=2,\ell=4$ and $k=6, \ell= 8$ respectively. %The upper plots are for $d=2$ and $k=1.5$ and $k=5.5$ and the lower plots are for $d=3$ and $k=2$ and $k=6$.
\newline
\begin{figure}[!ht]
    \centering
    \subfloat[$d=2,k=1.5$]{\includegraphics[width=0.45\linewidth]{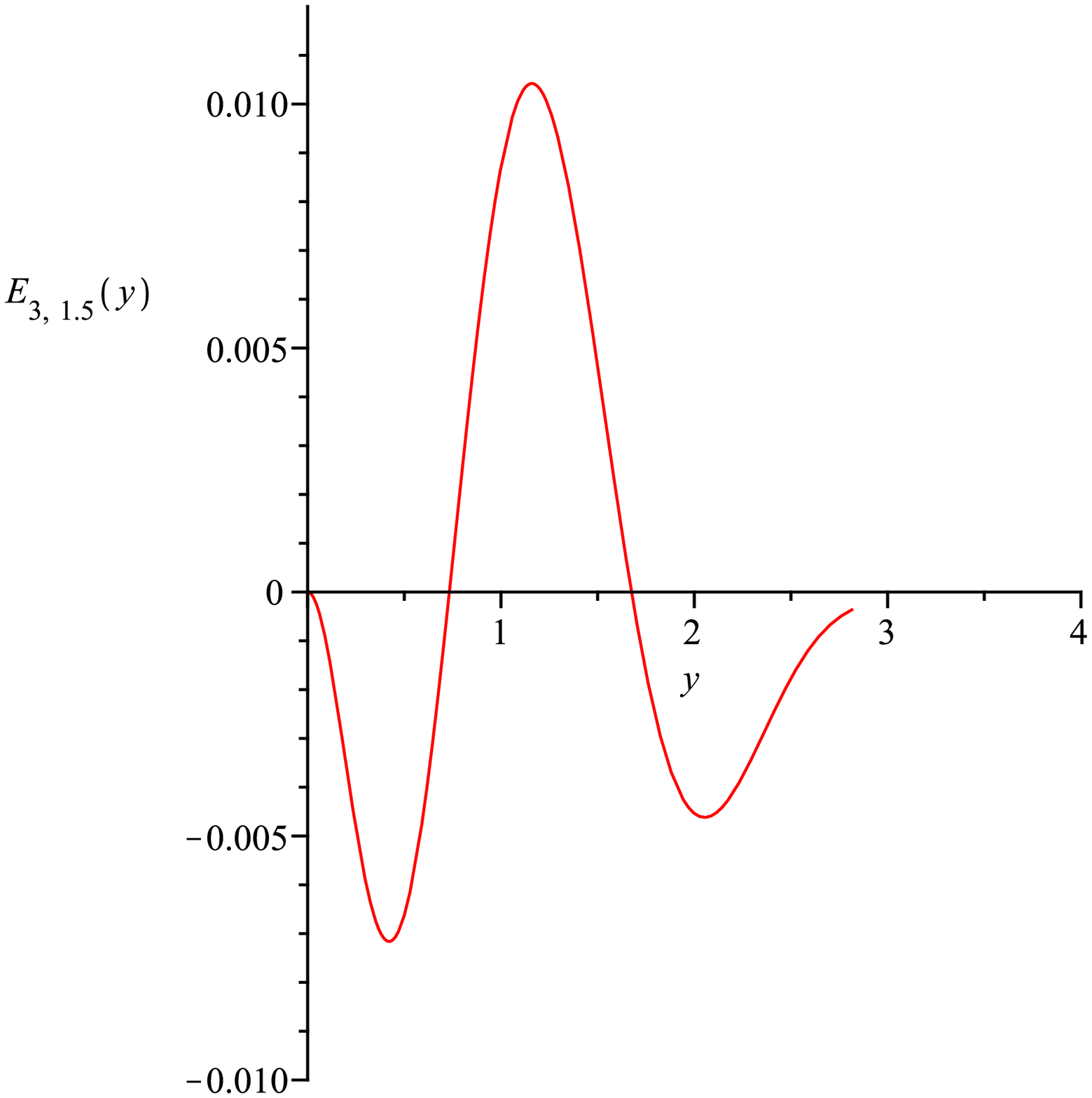}} \quad
    \subfloat[$d=2,k=5.5$]{\includegraphics[width=0.45\linewidth]{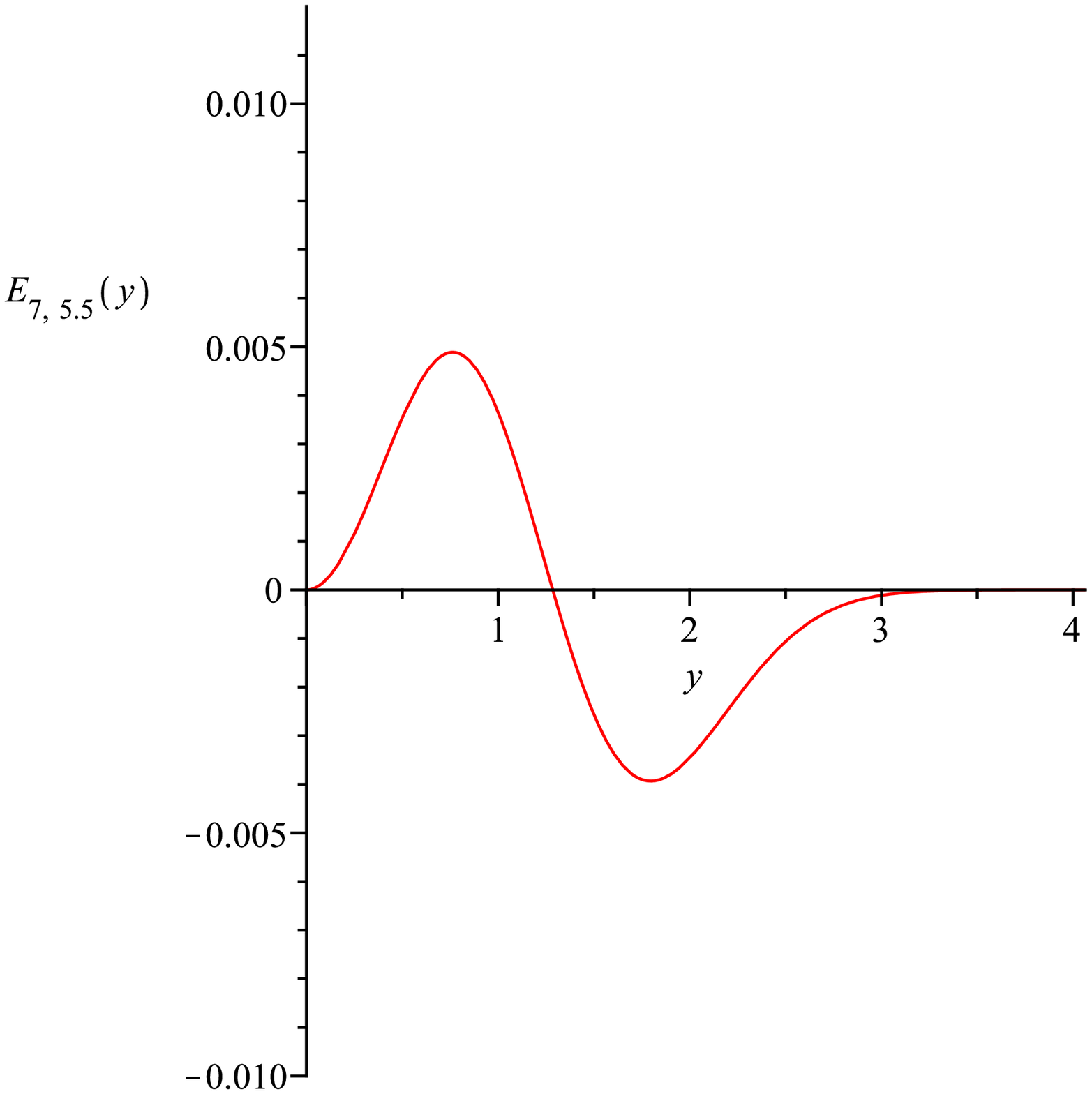}} \\
    \subfloat[$d=3,k=2$]{\includegraphics[width=0.45\linewidth]{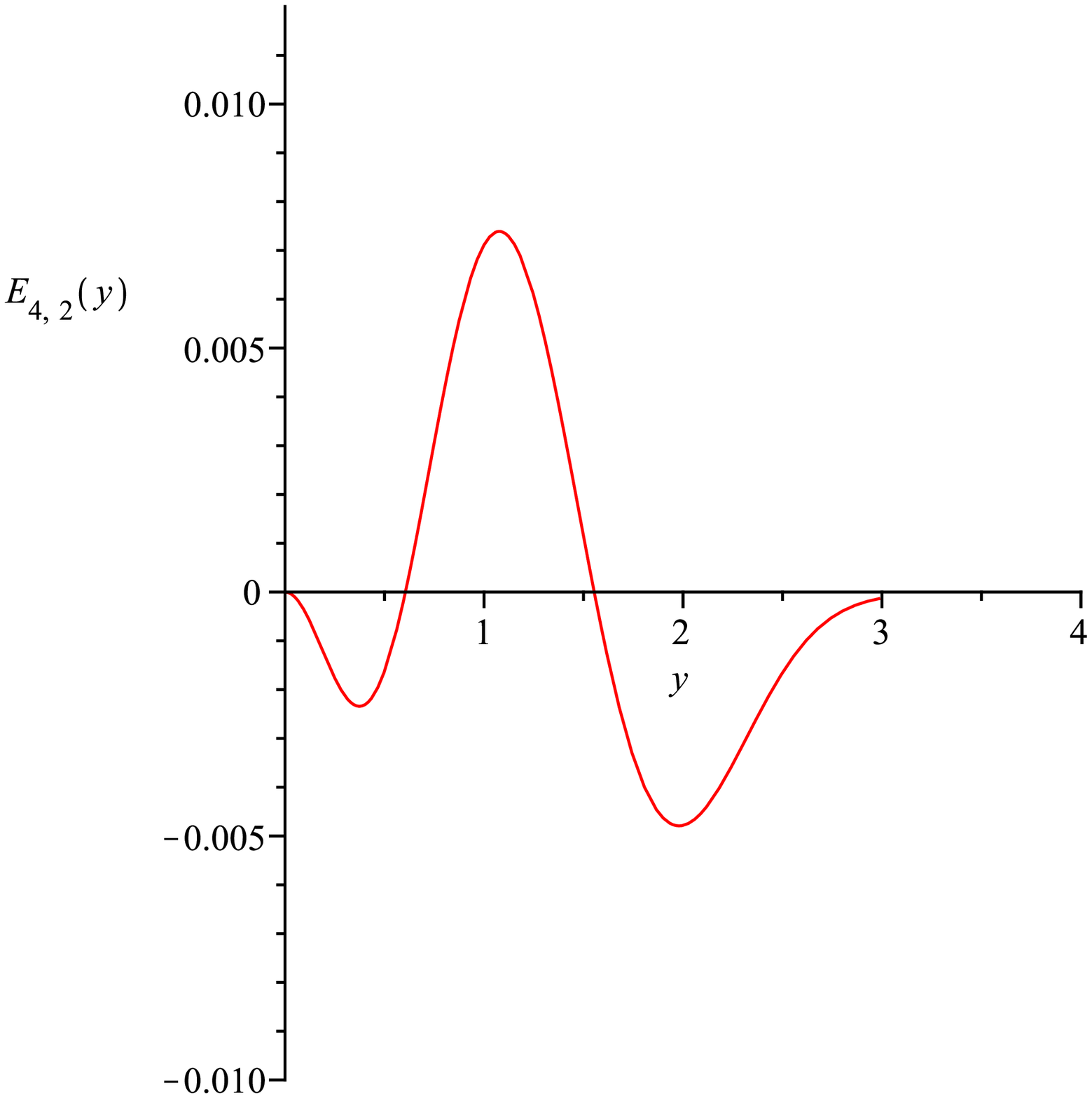}} \quad
    \subfloat[$d=3,k=6$]{\includegraphics[width=0.45\linewidth]{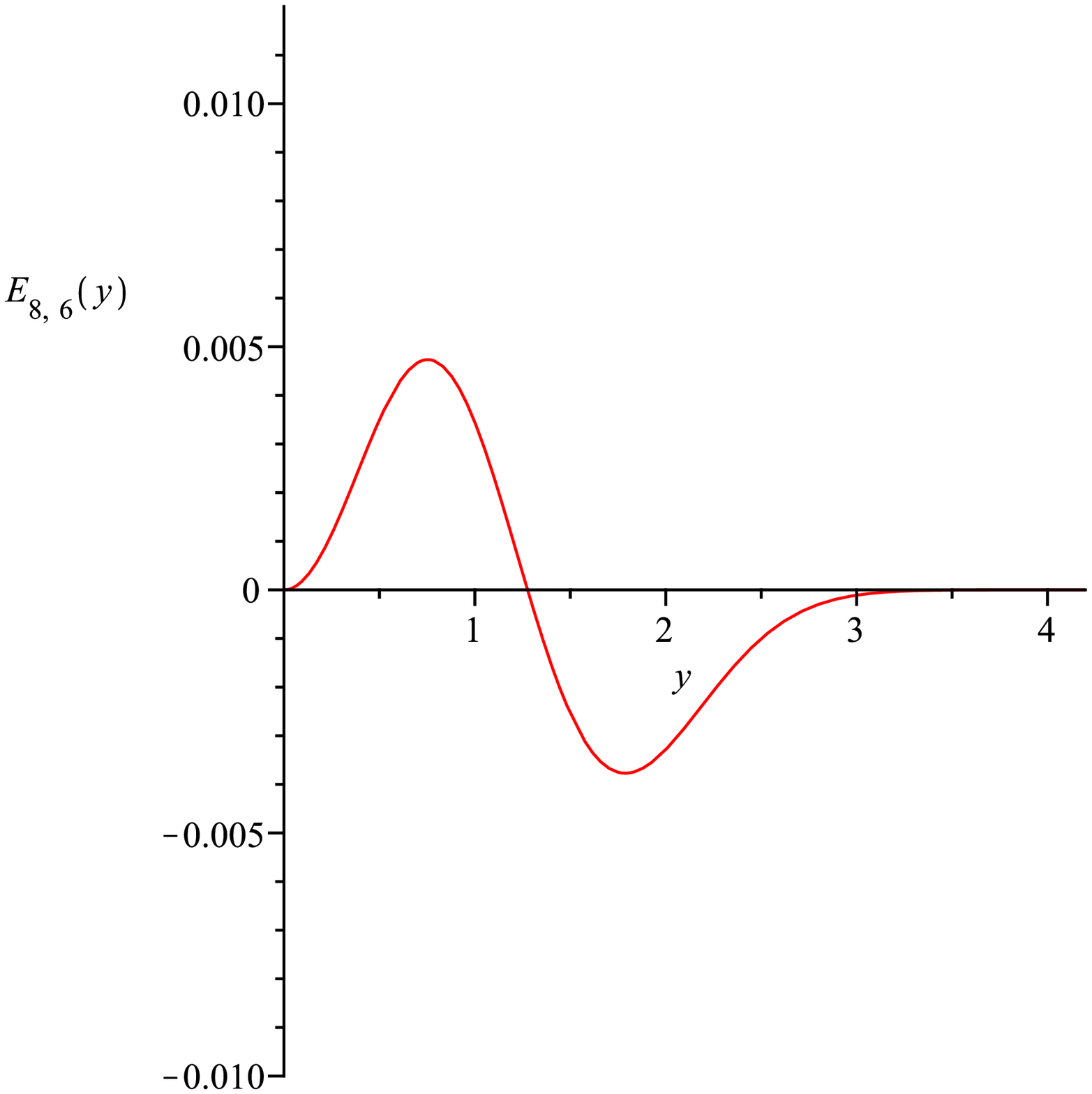}}
    \caption{$E_{\ell,k}(y)$ with $\alpha=1$ and $0 \leq y \leq \deltaLK(1)$. Subplots (a) and (b) are for the missing Wendland functions and subplots (c) and (d) are for the original Wendland functions.}
    \label{wfMinusGaussianGraphs}
\end{figure}

In the absence of theoretical rates of convergence, we show numerical results. Figure \ref{origWFsErrorGraphs} shows $\epsilon_{\ell,k}$ with $\alpha=1$ for $k=1,\ldots,50$ and $d=$ 3 and 5 for the original Wendland functions. Figure \ref{missWFsErrorGraphs} shows $\epsilon_{\ell,k}$ with $\alpha=1$ for $k = 0.5,\ldots,49.5$ and $d=$ 2 and 4 for the missing Wendland functions. Since $\alpha$ is just a scaling factor, the results do not vary in an essential way for different values of $\alpha$.
\newline
\begin{figure}[!ht]
    \centering
   \subfloat[$d=3$]{\includegraphics[width=0.45\linewidth]{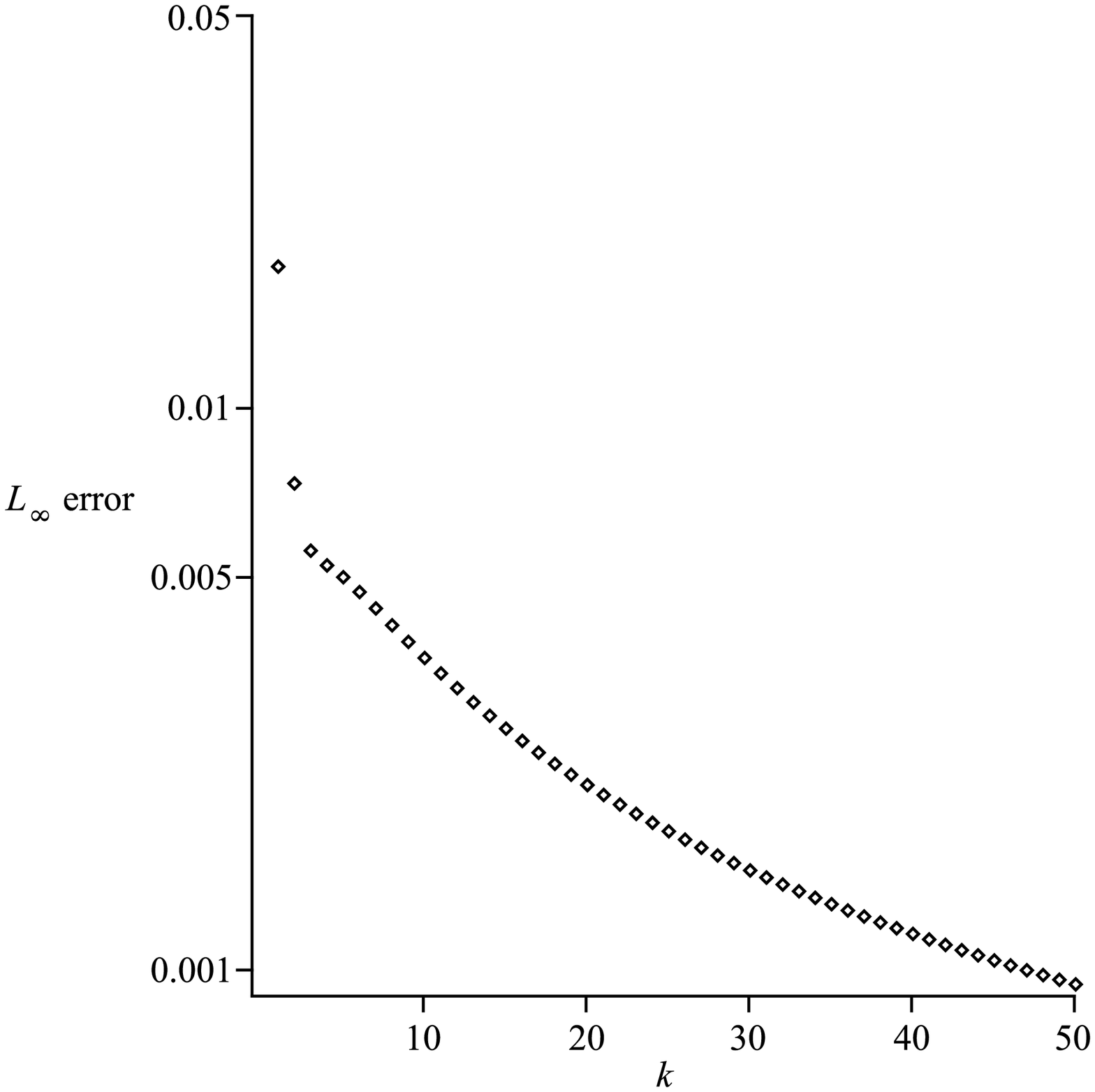}} \quad
   \subfloat[$d=5$]{\includegraphics[width=0.45\linewidth]{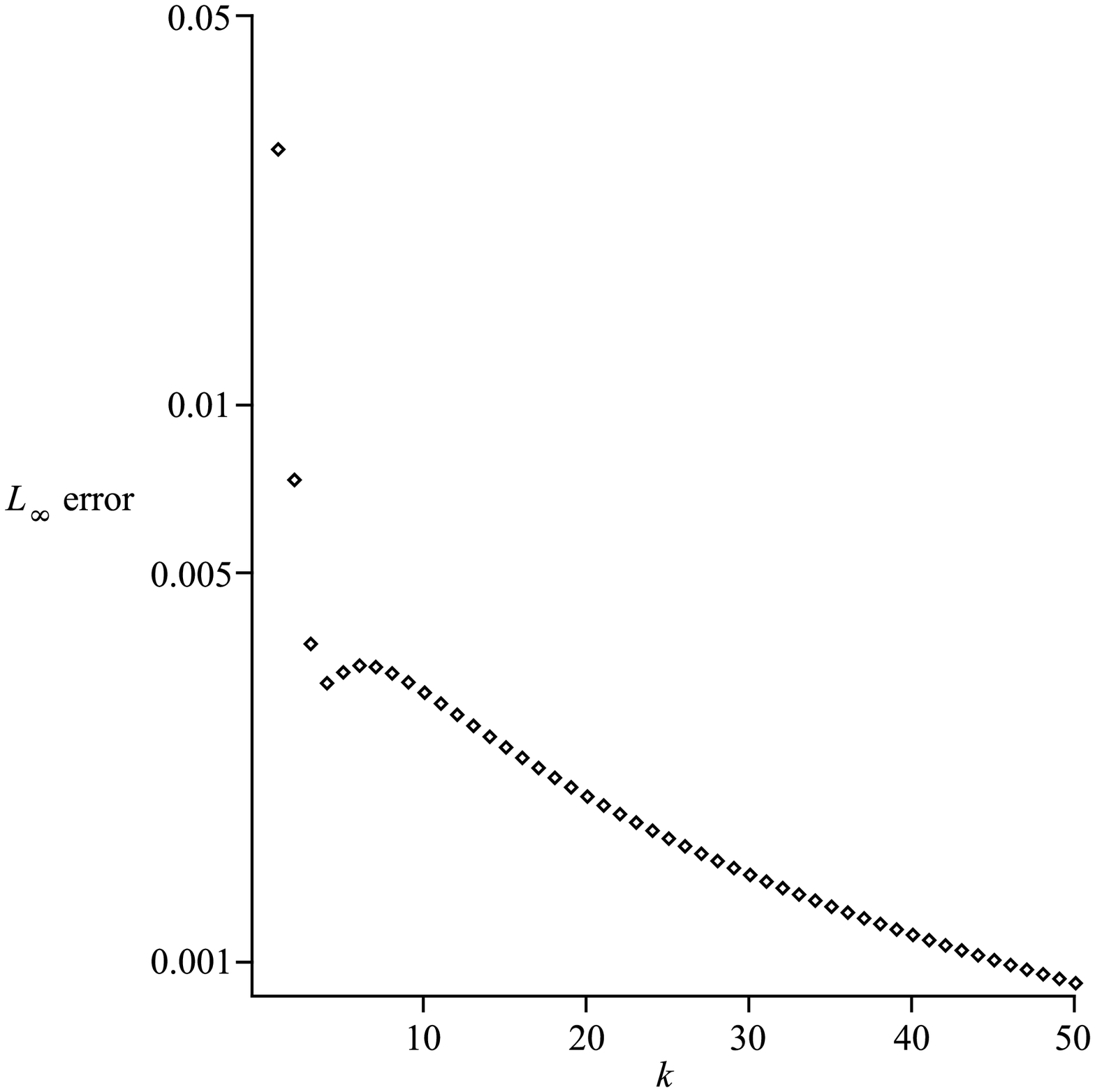}} \\
    \caption{$\epsilon_{\ell,k}$ on a logarithmic scale with $\alpha=1$, $k=1,\ldots,50$ and $d=3$ and 5 for the original Wendland functions.}
    \label{origWFsErrorGraphs}
\end{figure}

\begin{figure}[!ht]
    \centering
  \subfloat[$d=2$]{\includegraphics[width=0.45\linewidth]{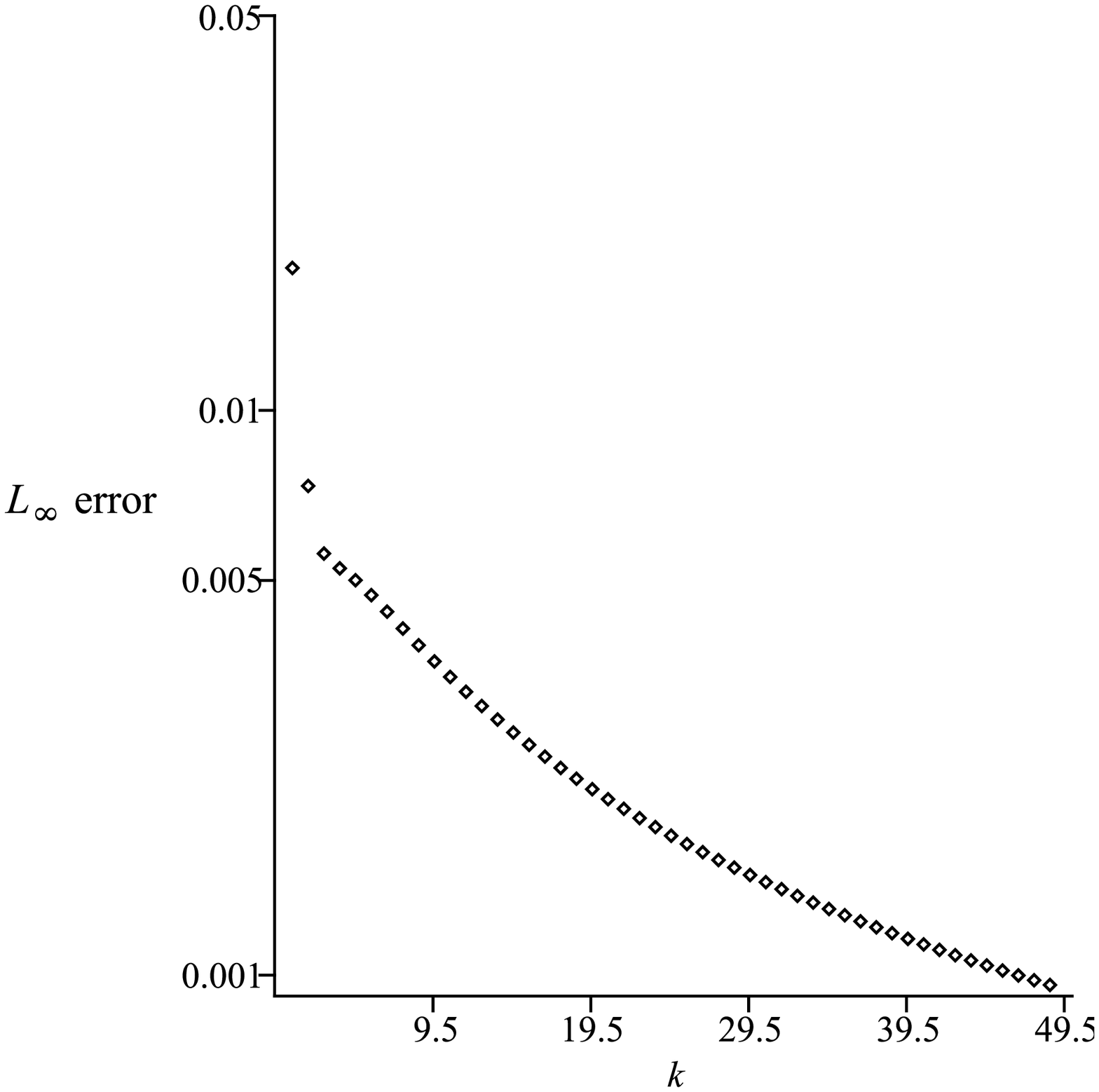}} \quad
  \subfloat[$d=4$]{\includegraphics[width=0.45\linewidth]{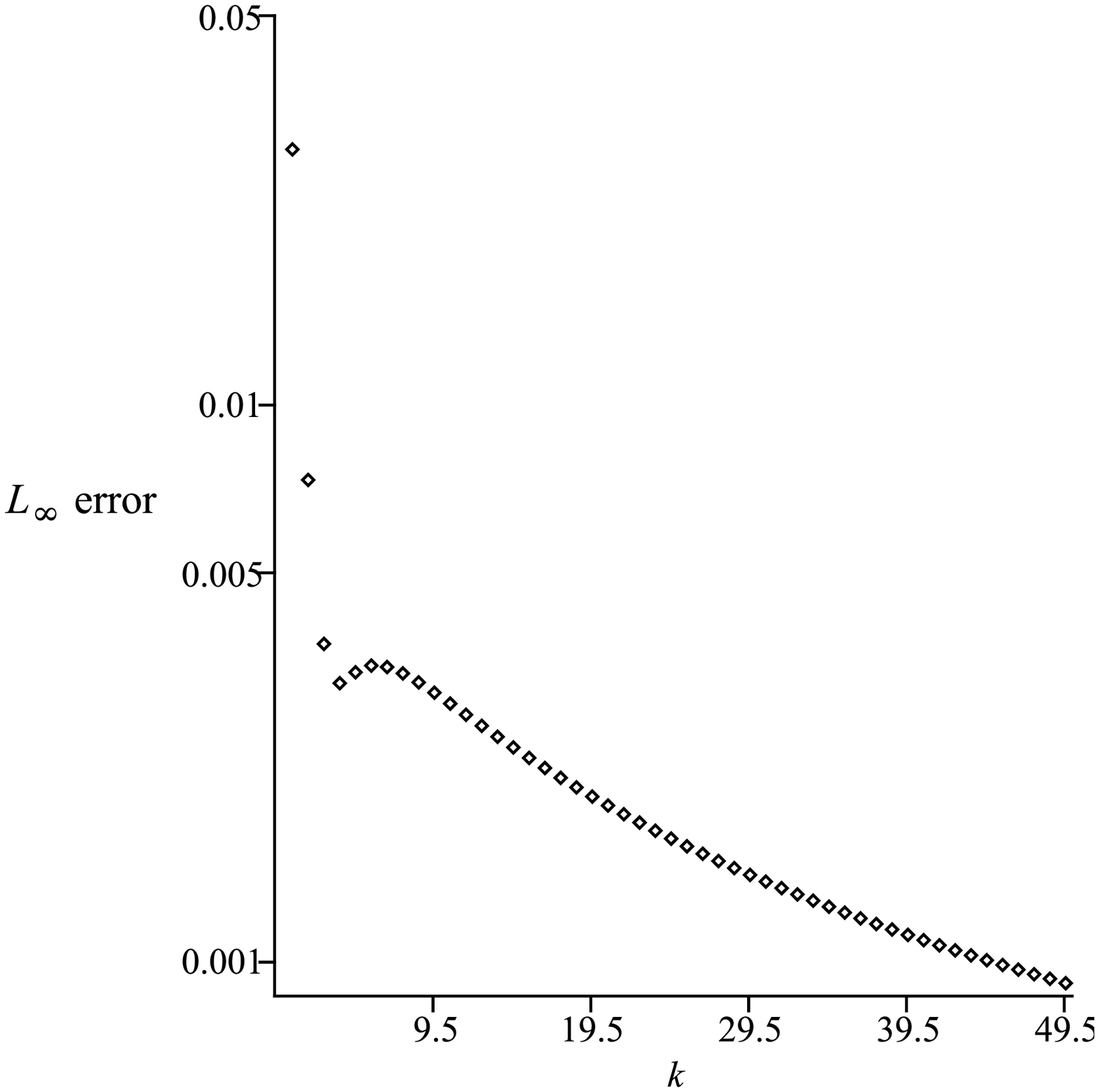}} \\
   \caption{$\epsilon_{\ell,k}$ on a logarithmic scale with $\alpha=1$, $k=0.5,\ldots,49.5$ and $d=2$ and 4 for the missing Wendland functions.}
    \label{missWFsErrorGraphs}
\end{figure}

In all cases, we see rapid convergence of $\epsilon_{\ell,k}$ to zero as the smoothness parameter $k$ increases. This is consistent with the theoretical convergence results. Note that $\epsilon_{\ell,k}$ is not monotonically decreasing in $k$. We also remark that $\epsilon_{\ell,k}$ is reached at different values of $y$ as $k$ increases.

\subsection{An interpolation example} \label{subsectionInterpolExample} \hfill
\newline
We consider an example, in which we show results obtained with both the Wendland functions $\phi_{\ell,k}$, normalised to have value 1 at the origin, and the normalised equal area Wendland functions $\psi_{\ell,k}$ for different values of $k$. The aim of the example is to approximate the 2-dimensional Franke function \cite[p.20]{Fra79} on $[0,5]^2$. For $k=1,\ldots,5$ we consider interpolation as in \eqref{eqnInterpolation1} and \eqref{eqnInterpolation2}, using the Wendland functions $\phi_{\ell,k}$ and the normalised equal area Wendland functions $\psi_{\ell,k}$ with $\alpha=2$. We use a $9 \times 9$ equally spaced grid as the centres. The number of centres is thus $n=81$. The $L_2$ error was estimated using Gaussian quadrature with a $120 \times 120$ tensor product grid of Gauss-Legendre points and the $L_{\infty}$ error was estimated by using a $360 \times 360$ equally spaced grid. Table \ref{frankeFnExample1} shows the $L_2$ and $L_{\infty}$ errors, as well as the 2-norm condition numbers of the interpolation matrices. We also show the results with the limiting Gaussian of $\exp(-2 \, y^2)$, denoted by $k=\infty$.
\newline

We see from the right-hand part of Table \ref{frankeFnExample1} that once the argument is properly scaled to give approximately constant effective support, increasing the smoothness has remarkably little effect on the error. On the other hand the condition number increases rapidly as the smoothness increases and is very large for the Gaussian limit. Taken together, these observations suggest that any benefit gained from the higher smoothness is likely to be offset by the increased condition numbers of the matrices.
\newline

The results with the Wendland functions $\phi_{\ell,k}$ are in the left-hand part of Table \ref{frankeFnExample1}. We can see that the condition number is decreasing as $k$ increases, which is due to the decreasing magnitude of the non-zero elements away from the diagonal. This is due to the fact that as $k$ increases the Wendland functions $\phi_{\ell,k}(r)$, normalised to have value 1 at $r=0$, decay more rapidly with respect to $r$, as illustrated in Figure \ref{FigOrigWFsd3k1to5}.
\begin{table}[!htbp]
\begin{centering}
\small\addtolength{\tabcolsep}{-2pt}
\begin{tabular}{|c|c| cc|c| cc| ccc|}
\hline
\multicolumn{2}{|c}{}& \multicolumn{3}{|c}{RBF: $\phi_{\ell,k}$} & \multicolumn{5}{|c|}{ RBF: $\psi_{\ell,k}$} \\
\hline
$N$  & $k$ & $L_2 $ error & $L_{\infty} $ error & $\kappa$ & $L_2 $ error & $L_{\infty} $ error & $\kappa$ & $\lambda_{\mathrm{min}}$ & $\lambda_{\mathrm{max}}$ \\
\hline
81 & 1 & 2.25e{-1} & 6.96e{-1} & 1.71 & 1.89e{-1} & 5.89e{-1} & 1.76e{1} & 1.55e{-1} & 2.74 \\
 & 2 & 2.61e{-1} & 7.95e{-1} & 1.22 & 1.86e{-1} & 5.78e{-1} & 3.14e{1} & 9.62e{-2} & 3.02\\
 & 3 & 3.00e{-1} & 8.90e{-1} & 1.07 & 1.87e{-1} & 5.79e{-1} & 4.96e{1} & 6.50e{-2} & 3.22\\
 & 4 & 3.36e{-1} & 9.73e{-1} & 1.02  & 1.87e{-1} & 5.80e{-1} & 5.56e{1} & 5.98e{-2} & 3.30 \\
& 5 & 3.63e{-1} & 1.03 & 1.01 & 1.87e{-1}  & 5.81e{-1} & 6.37e{1} &5.29e{-2} & 3.37\\
\hline
& $\infty$ & & && 1.89e{-1}&5.89e{-1} & 9.40e{1} & 4.03e{-2} & 3.78 \\
\hline
\end{tabular} \caption{Results from the example in Section \ref{subsectionInterpolExample} showing $L_2$ and $L_{\infty}$ errors, 2-norm condition numbers $\kappa$ and minimum and maximum eigenvalues ($\lambda_{\mathrm{min}}$ and $\lambda_{\mathrm{max}}$) when using the Wendland RBFs $\phi_{\ell,k}$ and the normalised equal area Wendland RBFs $\psi_{\ell,k}$ with $\alpha=2$.} \label{frankeFnExample1}
\end{centering}
\end{table}

\section{Conclusion} \label{SectionEffectiveSupport} \hfill
\newline

Compactly supported radial basis functions have proved very popular in scattered data approximation due to the resulting sparsity of the interpolation matrix. This paper has shown that for both the original and missing Wendland functions, the limit as the smoothness parameter goes to infinity, after suitable scaling and linear transformation, is a Gaussian RBF.
\newline

In Figure \ref{FigOrigWFsd3k1to5} we saw that the (original) normalised Wendland functions exhibit faster decay with respect to $r$ as the smoothness parameter $k$ increases. This suggests the need for a change of variable, not only to have a well-defined limit as considered in this paper, but perhaps also in practical applications. Without a change of variable, in the case of interpolation we could have a nearly diagonal interpolation matrix and consequently high errors between the interpolation points.
\newline

The results in the paper have illustrated the trade-off between approximation power and the condition number of the resulting linear system with Wendland functions of different smoothness. The issue of appropriate scaling and the selection of a smoothness parameter when using the Wendland functions remains a complex issue in practice.
\newline \newline
\textbf{Acknowledgements} Special thanks go to two referees for careful reading and useful suggestions.

\bibliographystyle{siam}
\bibliography{C:/Academic/PhD/LaTex/Bibliography/AndrewChernihPhDthesis}{}

\begin{thebibliography}{10}

\bibitem{Abr72}
{\sc M.~Abramowitz and I.~A. Stegun}, {\em Handbook of Mathematical Functions
  with Formulas, Graphs and Mathematical Tables}, vol.~65 of National Bureau of
  Standards Applied Mathematics Series, Dover Publications, 1972.

\bibitem{And00}
{\sc G.~E. Andrews, R.~Askey, and R.~Roy}, {\em Special Functions}, vol.~71 of
  Encylopedia of Mathematics and its Applications, Cambridge University Press,
  Cambridge, 2000.

\bibitem{Buh03}
{\sc M.~D. Buhmann}, {\em Radial Basis Functions}, vol.~12 of Cambridge
  Monographs on Applied and Computational Mathematics, Cambridge University
  Press, Cambridge, 2003.

\bibitem{DLMF}
{\sc DLMF}, {\em Digital Library of Mathematical Functions}, National Institute
  of Standards and Technology, 2011.

\bibitem{Fas07}
{\sc G.~E. Fasshauer}, {\em Meshfree Approximation Methods with MATLAB}, vol.~6
  of Interdisciplinary Mathematical Sciences, World Scientific Publishing Co.,
  Singapore, 2007.

\bibitem{FasZ07}
{\sc G.~E. Fasshauer and J.~G. Zhang}, {\em On choosing `optimal' shape
  parameters for {RBF} approximation'}, Numer. Algorithms, 45 (2007),
  pp.~345--368.

\bibitem{ForRS00}
{\sc M.~Fornefett, K.~Rohr, and H.S. Stiehl}, {\em Radial basis functions with
  compact support for elastic registration of medical images}, Image Vis.
  Comput., 19 (2001), pp.~87--96.

\bibitem{Fra79}
{\sc R.~Franke}, {\em A critical comparison of some methods for interpolation
  of scattered data}, Tech. Report NPS-53-79-003, Naval Postgraduate School,
  March 1979.

\bibitem{Gne99}
{\sc T.~Gneiting}, {\em Radial positive definite functions generated by
  {E}uclid's hat}, J. Multivar. Anal., 69 (1999), pp.~88--119.

\bibitem{Hub10}
{\sc S.~Hubbert}, {\em Closed form representations for a class of compactly
  supported radial basis functions}, Adv. Comput. Math., 36 (2012),
  pp.~115--136.

\bibitem{Mor08}
{\sc G.~Moreaux}, {\em Compactly supported radial covariance functions}, J.
  Geod., 82 (2008), pp.~431--443.

\bibitem{Rip99}
{\sc S.~Rippa}, {\em An algorithm for selecting a good value for the parameter
  $c$ in radial basis function interpolation}, Adv. Comput. Math, 11 (1999),
  pp.~193--210.

\bibitem{Sch95}
{\sc R.~Schaback}, {\em Creating surfaces from scattered data using radial
  basis functions}, in Mathematical Methods for Curves and Surfaces, M.~Daehlen
  T.~Lyche and L.L. Schumaker, eds., Vanderbilt University Press, Nashville,
  TN, 1995, pp.~477--496.

\bibitem{Sch09}
\leavevmode\vrule height 2pt depth -1.6pt width 23pt, {\em The missing
  {W}endland functions}, Adv. Comput. Math., 34 (2011), pp.~67--81.

\bibitem{SteW71}
{\sc E.M. Stein and G.~Weiss}, {\em Introduction to Fourier Analysis on
  Euclidean Spaces}, vol.~32 of Princeton Mathematical Series, Princeton
  University Press, Princeton, New Jersey, 1971.

\bibitem{Wen48}
{\sc J.~G. Wendel}, {\em Note on the gamma function}, Am. Math. Mon., 55
  (1948), pp.~563--564.

\bibitem{Wen95}
{\sc H.~Wendland}, {\em Piecewise polynomial, positive definite and compactly
  supported radial functions of minimal degree}, Adv. Comput. Math., 4 (1995),
  pp.~389--396.

\bibitem{Wen05}
\leavevmode\vrule height 2pt depth -1.6pt width 23pt, {\em Scattered Data
  Approximation}, vol.~17 of Cambridge Monographs on Applied and Computational
  Mathematics, Cambridge University Press, Cambridge, 2005.

\bibitem{Wu95}
{\sc Z.~Wu}, {\em Compactly supported positive definite radial basis
  functions}, Adv. Comput. Math., 4 (1995), pp.~283--292.

\bibitem{Zas06}
{\sc V.~P. Zastavnyi}, {\em On some properties of {B}uhmann functions}, Ukr.
  Math. J., 58 (2006), pp.~1045--1067.

\end{thebibliography}

\end{document}